\documentclass[12pt]{article}

\usepackage{amsmath,amsthm,amssymb}
\allowdisplaybreaks
\usepackage{mathrsfs}
\usepackage{cite}

\usepackage[pdftex,bookmarksopen=true,bookmarks=true,unicode,setpagesize]{hyperref}
\hypersetup{colorlinks=true,linkcolor=black,citecolor=black}

\usepackage{color}

\newtheorem{theorem}{Theorem}

\newtheorem{lemma}[theorem]{Lemma}
\newtheorem{proposition}[theorem]{Proposition}

\theoremstyle{remark}
\newtheorem{definition}[theorem]{Definition}
\theoremstyle{remark}

\theoremstyle{remark}
\newtheorem{remark}[theorem]{Remark}

\newcommand{\R}{{\mathbb R}}
\newcommand{\N}{{\mathbb N}}
\newcommand{\Z}{{\mathbb Z}}
\newcommand{\K}{{\mathbb K}}
\newcommand{\la}{\langle}
\newcommand{\ra}{\rangle}
\newcommand{\Rp}{{\R_+}}

\newcommand{\FCK}{\mathscr{FC}(\K(X))}
\newcommand{\FCG}{\mathscr{FC}(\Gamma_{pf}(\hat X))}

\renewcommand{\hat}{\widehat}
\begin{document}

\begin{center}{\Large \bf
Equilibrium diffusion  on  the cone of discrete Radon measures
}\end{center}

{\large Diana Conache}\\
 Fakult\"at f\"ur Mathematik, Universit\"at
Bielefeld, Postfach 10 01 31, D-33501 Bielefeld, Germany; \\
 e-mail:
\texttt{dputan@math.uni-bielefeld.de}\vspace{2mm}

{\large Yuri G. Kondratiev}\\
 Fakult\"at f\"ur Mathematik, Universit\"at
Bielefeld, Postfach 10 01 31, D-33501 Bielefeld, Germany;  NPU, Kyiv, Ukraine\\
 e-mail:
\texttt{kondrat@math.uni-bielefeld.de}\vspace{2mm}

{\large Eugene Lytvynov}\\ Department of Mathematics,
Swansea University, Singleton Park, Swansea SA2 8PP, U.K.\\
e-mail: \texttt{e.lytvynov@swansea.ac.uk}\vspace{2mm}

{\small

\begin{center}
{\bf Abstract}
\end{center}
\noindent Let $\mathbb K(\mathbb R^d)$ denote the cone of discrete Radon measures on $\mathbb R^d$. There is a natural  differentiation on $\mathbb K(\mathbb R^d)$: for a differentiable function $F:\mathbb K(\mathbb R^d)\to\mathbb R$, one defines its gradient $\nabla^{\mathbb K} F $ as a vector field  which assigns to each $\eta\in \mathbb K(\mathbb R^d)$ an element of a tangent space $T_\eta(\mathbb K(\mathbb R^d))$ to $\mathbb K(\mathbb R^d)$ at point $\eta$. Let $\phi:\mathbb R^d\times\mathbb R^d\to\mathbb R$ be a potential of pair interaction, and let $\mu$ be a corresponding  Gibbs perturbation of (the distribution of) a completely random measure on $\mathbb R^d$.
In particular, $\mu$ is a probability measure on $\mathbb K(\mathbb R^d)$ such that the set of atoms of a  discrete measure $\eta\in\mathbb K(\mathbb R^d)$ is $\mu$-a.s.\  dense in  $\mathbb R^d$. We consider the corresponding
Dirichlet form
$$
\mathscr E^{\mathbb K}(F,G)=\int_{\mathbb K(\mathbb R^d)}\langle\nabla^{\mathbb K} F(\eta), \nabla^{\mathbb K} G(\eta)\rangle_{T_\eta(\mathbb K)}\,d\mu(\eta).
$$
Integrating by parts with respect to the  measure $\mu$,  we explicitly find  the  generator of this Dirichlet form. By using the theory of Dirichlet forms, we prove the main result of the paper: If $d\ge2$,  there exists  a conservative diffusion process on $\mathbb K(\mathbb R^d)$ which is properly associated with the Dirichlet form $\mathscr E^{\mathbb K}$.} \vspace{2mm}

\vspace{2mm}


\newpage

\section{Introduction}

Let $X$ denote the Euclidean space $\R^d$ and let $\mathscr B(X)$ denote the Borel $\sigma$-algebra on $X$. Let $\mathbb M(X)$ denote the space of all Radon measures on $(X,\mathscr B(X))$. The space $\mathbb M(X)$ is equipped with the vague topology, and let $\mathscr B(\mathbb M(X))$ denote  the corresponding Borel $\sigma$-algebra on it.    A random measure  on $X$ is a
 measurable mapping $\xi:\Omega\to\mathbb M(X)$, where  $(\Omega,\mathscr F,P)$ is a probability space,
see e.g.\ \cite{Kallenberg}. A random measure $\xi$  is called  completely random if, for any mutually
 disjoint sets $A_1,\dots,A_n\in\mathscr B(X)$, the random variables $\xi(A_1),\dots,\xi(A_n)$ are independent \cite{Kingman}.

 The   cone of  discrete Radon measures on $X$ is defined by
$$\mathbb K(X):=\left\{
\eta=\sum_i s_i\delta_{x_i}\in \mathbb M(X) \,\Big|\, s_i>0,\, x_i\in X
\right\}.$$
 Here $\delta_{x_i}$ denotes the Dirac measure with mass at $x_i$.
 In the above representation, the atoms $x_i$ are assumed to be distinct and their total number is at most countable. By convention, the cone $\mathbb K(X)$ contains the null mass $\eta=0$, which is represented by the sum over an empty set of indices $i$.
 As shown in \cite{HKPR}, $\mathbb K(X)\in\mathscr B(\mathbb M(X))$. One endows $\K(X)$ with the vague topology.

 A random measure $\xi$ which takes values in $\K(X)$ with probability one is called a random discrete measure.
It follows from Kingman's result \cite{Kingman} that each completely random measure $\xi$ can be represented as
$\xi=\xi'+\eta$, where $\xi'$ is a deterministic measure on $X$ and $\eta$ is a random discrete measure. An important example of a random discrete measure is the gamma measure  \cite{TsVY}, which has many distinguished properties. It should be noted that, for a wide class of random discrete measures (including the gamma measure), the set of atoms of $\eta=\sum_i s_i\delta_{x_i}$, i.e., $\{x_i\}$, is dense in $X$.

In this paper, we will only use the distribution $\mu$ of a random discrete measure. So, below by a random discrete measure we will always mean  a probability measure $\mu$ on $(\K(X),\mathscr B(\K(X)))$. (Here $\mathscr B(\K(X))$ is the Borel $\sigma$-algebra on  $\K(X)$.)

In \cite{HKPR} Gibbs perturbations of the gamma measure were constructed, and in \cite{Conache} this result was extended to Gibbs perturbations of a general completely random discrete measure.
More precisely, let $\phi:X\times X\to\R$  be a  potential of pair interaction, which
 satisfies the conditions (C1), (C2) below. In particular, it is assumed that the function $\phi$ is symmetric, bounded, has finite range (i.e., $\phi(x,x')=0$ if the distance between $x$ and $x'$ is sufficiently large), and the positive part of $\phi$ dominates, in a sense, its negative part. For  $\eta\in\K(X)$, we heuristically define the energy of $\eta$ (Hamiltonian) by $$ H(\eta):=\frac12\int_{X^2\setminus D}\phi(x,x')\,d\eta(x)\,d\eta(x'),$$
 where $D=\{(x,x')\in X^2\mid x=x'\}$. Let $\nu$ be a completely random discrete measure. The Gibbs perturbation of  $\nu$
corresponding to the potential $\phi$ is heuristically defined as a probability measure $\mu$ on $\K(X)$ given by
$$d\mu(\eta):=\frac1Z\,e^{-H(\eta)}\,d\nu(\eta),$$
where $Z$ is a normalizing factor. A rigorous definition of $\mu$ is given  through the Dobrushin--Lanford--Ruelle equation. It is proven in \cite{HKPR} that such a Gibbs measure exists. In \cite{Conache}, it was shown that such a Gibbs measure is unique, provided the supremum norm of $\phi$, i.e., $\|\phi\|_\infty$, and the first moment of $\nu$ are sufficiently small.  In the general case, the uniqueness problem is still open.

 Any Gibbs measure $\mu$ satisfies the Nguyen--Zessin identity in which the relative energy of interaction between a single atom  measure $\eta=s\delta_x$ and a discrete measure $\eta'\in\K(X)$, with no atom at $x$, is given by $$H(\eta\mid\eta')=s\int_{X}\phi(x,x')\,d\eta'(x').$$ 

In \cite{KLV} (see also \cite{HKLV}), some elements of differential geometry on $\K(X)$ were introduced.
In particular, for a differentiable function $F:\K(X)\to\R$, one defines its gradient $\nabla^\K F $ as a vector field
 which assigns to each $\eta\in \K(X)$ an element of a tangent space $T_\eta(\K(X))$ to $\K(X)$ at point $\eta$. It should be stressed that $\K(X)$ is not a flat space, in the sense that the tangent space $T_\eta(\K)$ changes with a change of $\eta$.

So, in this paper, we consider the Dirichlet form
\begin{equation}\label{hfyrf7}
\mathscr E^\K(F,G):=\int_{\K(\R^d)}\big\la\nabla^\K F(\eta), \nabla^\K G(\eta)\big\ra_{T_\eta(\K)}\,d\mu(\eta).
\end{equation}
This bilinear form is initially defined on an appropriate set of smooth cylinder functions on $\K(X)$.
Using the Nguyen--Zessin identity, we carry out integration by parts with respect to the Gibbs measure $\mu$, and  find the $L^2$-generator of the bilinear form $\mathscr E^\K$ (containing the potential $\phi$ and its gradient). This, in particular, proves the closability of the bilinear form $\mathscr E^\K$ on $L^2(\K(X),\mu)$. This result extends \cite{KLV} (see also \cite{HKLV}), where the $L^2$-generator of $\mathscr E^\K$ (the Laplace operator) was derived in the case of no interaction, $\phi=0$, and when the completely random measure $\mu=\nu$ is  the law of a measure-valued L\'evy process.

The main result of the paper is the existence of a conservative diffusion process on $\K(X)$ which is properly associated with the Dirichlet form $\mathscr E^\K$. For this, one  assumes that the dimension of the underlying space $X$ is $\ge2$. (It is intuitively clear that in the case where the dimension of $X$ is equal to one, such a result should fail.)
We note that this diffusion process  has continuous sample paths in $\K(X)$ with respect to the vague topology.
The diffusion process has $\mu$ as invariant (and even symmetrizing) measure.
To prove the main result, we use the general theory of Dirichlet forms \cite{MR1} as well as the theory of Dirichlet forms over configuration spaces \cite{MR2,RS2}, see also \cite{AKR,KLR}.

The paper is organized as follows. In Section~\ref{fgxtdrst}, we recall how differentiation on $\K(X)$ is introduced  \cite{KLV}, and how the Gibbs measure $\mu$ is constructed \cite{HKPR,Conache}. In Section~\ref{hugfu7tr}, we formulate
the results of the paper. Finally,  Section~\ref{ft67r67} contains the proofs.

\section{Preliminaries} \label{fgxtdrst}

\subsection{Differentiation on $\K(X)$}

In this subsection, we follow \cite{KLV}. A starting point to define  differentiation on $\K(X)$ is the choice of  a natural group $\mathfrak G$ of transformations of $\K(X)$. So let $\operatorname{Diff}_0(X)$ denote the group of $C^\infty$ diffeomorphisms of $X$ which are equal to the identity outside a compact set. Let $C_0(X\to\R_+)$ denote the multiplicative group of continuous functions on $X$ with values in $\R_+:=(0,\infty)$ which are equal to one outside a compact set. The group $\operatorname{Diff}_0(X)$ naturally acts on $X$, hence on $C_0(X\to\Rp)$. So we define a group $\mathfrak G$ by
$$\mathfrak G:=\operatorname{Diff}_0(X) \rightthreetimes C_0(X\to\R_+),$$
the semidirect product of $\operatorname{Diff}_0(X)$ and
$C_0(X\to\Rp)$. As a set, $\mathfrak G$ is equal to the Cartesian product of   $\operatorname{Diff}_0(X)$ and $C_0(X\to\Rp)$, and the product in $\mathfrak G$ is given by
$$g_1g_2=(\psi_1\circ\psi_2,\, \theta_1(\theta_2\circ\psi_1^{-1}))\quad\text{for $g_1=(\psi_1,\theta_1),\, g_2=(\psi_2,\theta_2)\in \mathfrak G$.}$$
 The group $\mathfrak G$  naturally acts on $\K(X)$:  for any $g=(\psi,\theta)\in\mathfrak G$ and any $\eta\in\mathbb K(X)$, we define  $g\eta\in\K(X)$ by
$$
 d(g\eta)(x):=\theta(x)\,d(\psi^*\eta)(x).$$
Here $\psi^*\eta$ is the pushforward of $\eta$ under $\psi$.

The Lie algebra of the Lie group $\operatorname{Diff}_0(X)$ is the space $\operatorname{Vec}_0(X)$ consisting of all smooth vector fields acting from $X$ into $X$  which have compact support. For $v\in \operatorname{Vec}_0(X)$, let $(\psi_t^v)_{t\in\R}$ be the corresponding  one-parameter subgroup of $\operatorname{Diff}_0(X)$,  see e.g.\ \cite{Boo}.
As the Lie algebra of $C_0(X\to\Rp)$ we may take the space $C_0(X)$ of all real-valued continuous functions on $X$ with compact support. For each $h\in C_0(X)$, the corresponding one-parameter subgroup of $C_0(X\to\Rp)$ is given by $(e^{th})_{t\in\R}$.
Thus, $\mathfrak g:= \operatorname{Vec}_0(X)\times C_0(X) $ can be thought of as a Lie algebra that corresponds to the Lie group $\mathfrak G$. For an arbitrary $(v,h)\in\mathfrak g$, we may consider the curve $\{(\psi_t^v,e^{th}),\,t\in\R\}$ in $\mathfrak G$. For a  function $F:\K(X)\to\R$ we define its derivative in direction $(v,h)$ by
$$\nabla^{\K}_{(v,h)}F(\eta):=\frac{d}{dt}\Big|_{t=0}F((\psi_t^v,e^{th})\eta),\quad \eta\in\K(X),$$
provided the derivative on the right hand side of this formula exists.

A tangent space to $\K(X)$ at $\eta\in\K(X)$ is defined by
\begin{equation}\label{hjfuy8y}
T_\eta(\K(X)):=L^2(X\to X\times\R,\eta),\end{equation}
the $L^2$-space of $X\times\R$-valued vector fields  on $X$ which are square integrable with respect to the measure $\eta$.
We then define a gradient of a differentiable function $F:\K(X)\to\R$ at $\eta$ as the  element $(\nabla^{\K}F)(\eta)$ of $T_\eta(\K)$ which satisfies
$$\nabla^{\K}_{(v,h)}F(\eta)=\langle \nabla^{\K}F(\eta),(v,h)\rangle_{T_\eta(\K)}\quad\text{for all }(v,h)\in\mathfrak g.$$

\begin{remark}
Note that, in  the above definitions, one could replace $\K(X)$ with the wider space $\mathbb M(X)$. This is why, in paper \cite{KLV}, the gradient $\nabla^\K$ was actually denoted by $\nabla^{\mathbb M}$.
\end{remark}

Let us now define a set of test functions on $\K(X)$.  Let us denote by $\tau(\eta)$ the set of atoms of $\eta$, and for each $x\in\tau(\eta)$, let $s_x:=\eta(\{x\})$. Thus, we have
$$\eta=\sum_{x\in\tau(\eta)}s_x\delta_x.$$

We define a metric on $\Rp$ by
$$ d_{\Rp}(s_1,s_2):=\left|\log(s_1)-\log(s_2)\right|,\quad s_1,s_2\in\Rp.$$
Then $\Rp$ becomes a locally compact Polish space, and any set of the form $[a,b]$, with $0<a<b<\infty$, is compact.
We denote $\hat X:=\Rp\times X$, and let $C_0^\infty(\hat X)$ denote the space of all smooth functions on $\hat X$ with compact support.
 For each $\varphi\in C_0^\infty(\hat X)$ and $\eta\in\K(X)$, we define
$$\langle\!\langle \varphi,\eta\rangle\!\rangle:=\sum_{x\in\tau(\eta)}\varphi(s_x,x).$$
Note that the latter sum contains only finitely many nonzero terms.

We denote by $\FCK$ the set of all  functions $F:\K(X)\to\R$ of the form
\begin{equation}\label{hufutr}
 F(\eta)=g\big(\langle\!\langle \varphi_1,\eta\rangle\!\rangle,\dots,\langle\!\langle \varphi_N,\eta\rangle\!\rangle\big),\quad \eta\in\K(X),\end{equation}
where $g\in C_{b}^\infty(\R^N)$, $\varphi_1\,\dots,\varphi_N\in C_0^\infty(\hat X)$, and $N\in\N$.  Here $C_{b}^\infty(\R^N)$ is the set of all infinitely differentiable functions on $\R^N$ which, together with all their derivatives, are bounded.

Let $F:\K(X)\to\R$, $\eta\in\K(X)$, and $x\in\tau(\eta)$. We define
\begin{align}
\nabla_xF(\eta):=&\nabla_y\big|_{y=x}F(\eta-s_x\delta_x+s_x\delta_y),\label{gfuyfr}\\
\nabla_{s_x}F(\eta):=&\frac{d}{du}\Big|_{u=s_x}F(\eta-s_x\delta_x+u\delta_x),\label{vyd6trde}
\end{align}
provided the derivatives  exist. Here the variable $y$ is from $X$,  $\nabla_y$ denotes the  gradient on $X$ in the $y$ variable, and
the variable $u$ is from $\Rp$.

An easy calculation shows that, for each function $F\in\FCK$, the gradient $\nabla^\K F$ exists and is given by
\begin{equation}\label{guft}
 (\nabla^{\K}F)(\eta,x)=\left(\frac1{s_x}\nabla_x F(\eta),\nabla_{s_x}F(\eta)\right),\quad \eta\in\K(X),\  x\in\tau(\eta).\end{equation}

\subsection{The Gibbs measures}

We start with defining a class of completely random measures. Let $l:\hat X\to\Rp$ be a measurable function which satisfies the following conditions: for $dx$-a.a.\ $x\in X$
\begin{equation}\label{hfjiu}
\int_{\Rp}\frac{l(s,x)}{s}\,ds=\infty\end{equation}
and for each $\Lambda\in\mathscr B_0(X)$,
\begin{equation}\label{jkufr7u}
\int_{\Rp\times\Lambda}l(s,x)\,ds\,dx<\infty.\end{equation}
Here $\mathscr B_0(X)$ denotes the collection of all sets from $\mathscr B(X)$ which have compact closure.

We define a measure $\sigma$ on $\hat X$ by
\begin{equation}\label{gfytre}
 d\sigma(s,x):=\frac{l(s,x)}{s}\,ds\,dx.
\end{equation}
 Since \eqref{jkufr7u} holds, we may define a completely random measure $\nu$ as a probability measure on $\K(X)$ which has Fourier transform
$$
 \int_{\K(X)}e^{i\langle f,\eta\rangle}\,d\nu(\eta)=\exp\left[\int_{\hat X} (e^{isf(x)}-1)\,d\sigma(s,x)\right],\quad f\in C_0(X),$$
see e.g.\ \cite{DVJ2}.
Here we denote $\langle f,\eta\rangle:=\int_X f(x)\,d\eta(x)$. The measure $\nu$ can also be characterized through the Mecke identity: $\nu$ is the unique probability measure on $\K(X)$ which satisfies, for each measurable function $F:\hat X\times\K(X)\to[0,\infty]$,
\begin{equation}\label{Mecke}
\int_{\K(X)}\sum_{x\in\tau(\eta)}F(s_x,x,\eta)\,d\nu(\eta)=\int_{\K(X)}d\nu(\eta)\int_{\hat X}d\sigma(s,x)\,F(s,x,\eta+s\delta_x).
\end{equation}

For example, by choosing $l(s,x)=e^{-s}$, we get the gamma measure $\nu$ \cite{TsVY}. More generally, we may fix measurable functions $\alpha,\beta:X\to\Rp$ and set
$$l(s,x)=\beta(x)e^{-s/\alpha(x)}.$$
Then conditions \eqref{hfjiu}, \eqref{jkufr7u} are satisfied when $\alpha(x)\beta(x)\in L^1_{\mathrm{loc}}(X,dx)$.

Let us now recall the definition of a Gibbs measure from \cite{HKPR,Conache}.
Additionally to \eqref{hfjiu} and \eqref{jkufr7u}, we assume that, for each $\Lambda\in\mathscr B_0(X)$,
\begin{equation}\label{khgiygt}
\int_{\Rp\times\Lambda}l(s,x)s\,ds\,dx<\infty.\end{equation}

Let $\phi: X\times X\to\R$ be a pair potential which satisfies the following two conditions:

\begin{itemize}

\item[(C1)] $\phi$ is a symmetric, bounded, measurable function which satisfies, for some $R>0$,
$$\phi(x,y)=0\quad\text{if }|x-y|>R.$$

\item[(C2)] There exists $\delta>0$ such that
$$\inf_{x,y\in X:\,\,
|x-y|\le\delta} \phi(x,y)>\varepsilon \|\phi^-\|_\infty.$$
Here
$$\|\phi^-\|_\infty:=\sup_{x,y\in X}(-\phi(x,y)\vee 0)$$
 and
$\varepsilon:= 2v_d\, d^{d/2}(R/\delta+1)$, where $v_d:=\pi^{d/2}/\Gamma(d/2+1)$ is the volume of a unit ball in $X$.
\end{itemize}

\begin{remark}
Note that condition (C2) excludes the  potential $\phi=0$. Note also that
  conditions (C1) and (C2) are trivially satisfied if $\phi(x,y)=\psi(x-y)$, where $\psi\in C_0(X)$, $\psi(x)=\psi(-x)$, and $\psi(0)>v_d\, d^{d/2}\|\psi^-\|_\infty$.\end{remark}

For any $\eta,\xi\in\K(X)$ and $\Lambda\in\mathscr B_0(X)$, we define the relative energy (Hamiltonian)
$$ H_\Lambda(\eta\mid\xi):=\frac12\int_{\Lambda^2\setminus D}\phi(x,y)\,d\eta(x)\,d\eta(y)+\int_{\Lambda^c}\int_\Lambda
\phi(x,y)\,d\eta(x)\,d\xi(y),$$
where $\Lambda^c:=X\setminus\Lambda$. Note that $H_\Lambda(\eta\mid\xi)$ is well defined and finite.

For each $\Lambda\in\mathscr B(X)$, we denote $\K(\Lambda):=\{\eta\in\K(X)\mid \tau(\eta)\subset\Lambda\}$. Note that $\K(\Lambda)\in\mathscr B(\K(X))$. Let $\nu_\Lambda$ denote the pushforward of the completely random measure $\nu$ under the canonical projection
$$\K(X)\ni\eta\mapsto \eta_\Lambda:=\sum_{x\in\tau(\eta)\cap\Lambda}s_x\delta_x\in\K(\Lambda).$$
The measure $\nu_\Lambda$ has Fourier transform
$$  \int_{\K(\Lambda)}e^{i\langle f,\eta\rangle}\,d\nu_\Lambda(\eta)=
\exp\left[\int_{\Rp\times\Lambda} (e^{isf(x)}-1)\,d\sigma(s,x)\right],\quad f\in C_0(X).$$

\begin{proposition}[\cite{HKPR,Conache}]\label{hguf7tuo}
Let \eqref{hfjiu}--\eqref{gfytre}, \eqref{khgiygt} hold and let conditions {\rm (C1)} and {\rm (C2)} be satisfied. Then, for any $\Lambda\in\mathscr B_0(X)$ and $\xi\in\K(X)$,
$$0<Z_\Lambda(\xi):=\int_{\K(\Lambda)}e^{-H(\eta\,|\,\xi)}\,d\nu_\Lambda(\eta)<\infty.$$
\end{proposition}

For each $\Lambda\in\mathscr B_0(X)$ with $\int_\Lambda dx>0$, the local Gibbs state with boundary condition $\xi\in\K(X)$ is defined as a probability measure on $\K(\Lambda)$ given by
$$d\mu_\Lambda(\eta\mid\xi):=\frac1{Z_\Lambda(\xi)}\,e^{-H(\eta\,|\,\xi)}\,d\nu_\Lambda(\eta).$$
For each $B\in\mathscr B(\K(X))$, $\Lambda\in\mathscr B_0(X)$, and $\xi\in\K(X)$, we define
$$B_{\Lambda,\xi}:=\{\eta\in\K(\Lambda)\mid \eta+\xi_{\Lambda^c}\in B\}\in\mathscr B(\K(\Lambda))$$
and hence we can define the local specification $\Pi=\{\pi_{\Lambda}\}_{\Lambda\in\mathscr B_0(X)}$ on $\K(X)$ as the family of stochastic kernels
$$\mathscr{B}(\K(X))\times \K(X)\ni (B,\xi) \mapsto \pi_{\Lambda}(B\mid \xi)\in[0,1]$$
given by $\pi_{\Lambda}(B\mid \xi):=\mu_{\Lambda}(B_{\Lambda,\xi})$.

\begin{definition}
A Gibbs perturbation of a completely random measure $\nu$  corresponding to a pair potential $\phi$  is defined as a probability measure $\mu$ on $(\K(X),\mathscr B(\K(X)))$ which satisfies the following
Dobrushin--Lanford--Ruelle (DLR) equation:
\begin{equation}\label{DLR}
\int_{\K(X)}\pi_\Lambda(B\mid\xi)\,d\mu(\xi)=\mu(B),
\end{equation}
for any $B\in\mathscr B(\K(X))$ and $\Lambda\in\mathscr B_0(X)$. We denote by $G(\nu,\phi)$ the set of all such probability measures $\mu$.
\end{definition}

\begin{theorem}[\cite{HKPR,Conache}]\label{bvufufr7}
Let the conditions of Proposition \ref{hguf7tuo} be satisfied. Then the set $G(\nu,\phi)$ is non-empty. Furthermore, each measure $\mu\in G(\nu,\phi)$ has finite moments: for each $\Lambda\in\mathscr B_0(X)$ and $n\in\N$,
\begin{equation}\label{iyutfr76r}
\int_{\K(X)}\eta(\Lambda)^n\,d\mu(\eta)<\infty.\end{equation}
\end{theorem}

Since \eqref{hfjiu} holds, for each $\Lambda\in\mathscr B_0(X)$ with $\int_\Lambda dx>0$, for $\nu$-a.a.\ $\eta\in\K(X)$, the set $\tau(\eta)\cap\Lambda$ is infinite. Using the DLR equation, we therefore obtain the following result.

\begin{proposition} Let the conditions of Proposition \ref{hguf7tuo} be satisfied, and let $\mu\in G(\nu,\phi)$. Let $\Lambda\in\mathscr B_0(X)$ with $\int_\Lambda dx>0$. Then, for $\mu$-a.a.\ $\eta\in\K(X)$, the set $\tau(\eta)\cap\Lambda$ is infinite.  In particular, the set $\tau(\eta)$ is $\mu$-a.s.\ dense in $X$.
\end{proposition}

By analogy with \cite{NZ}, the Gibbs measures have the following property.

\begin{theorem}\label{hjugfuyt8} Let the conditions of Proposition \ref{hguf7tuo} be satisfied, and let $\mu\in G(\nu,\phi)$.
Then $\mu$ satisfies the following  Nguyen--Zessin identity: for each measurable function $F:\hat X\times\K(X)\to[0,\infty]$,
\begin{multline}
\int_{\K(X)}\sum_{x\in\tau(\eta)}F(s_x,x,\eta)\,d\mu(\eta)\\
=\int_{\K(X)}\int_{\hat X}\exp\left[-s\int_{X}\phi(x,x')\,d\eta(x')\right]F(s,x,\eta+s\delta_x)\,d\sigma(s,x)d\mu(\eta).\label{ghydcydr}\end{multline}
\end{theorem}

\begin{proof} 
By the same arguments as in the proof of \cite[Theorem~6.3]{HKPR}, it is enough to show that, for each $\Lambda\in \mathscr B_0(X)$, equality 
(\ref{ghydcydr}) holds for all functions $F$ of the form $F(s,x,\eta)=f(s,x)g(\eta_{\Lambda})$, where $f\in C_{0}(\hat{X})$, $f\ge0$, the support of $f$ is a subset of  $\mathbb{R}_{+}\times \Lambda$ and $g:\K(\Lambda)\to[0,\infty)$ is bounded and measurable. 
By the DLR equation (\ref{DLR}) and the Mecke identity (\ref{Mecke}), we have
\begin{align}&
\int_{\K(X)}\sum_{x\in\tau(\eta)}F(s_x,x,\eta)\,d\mu(\eta) = \int_{\K(X)}\int_{\K(X)}\sum_{x\in\tau(\eta)\cap\Lambda}f(s_x,x)g(\eta)\, \pi_{\Lambda}(d\eta\mid\xi)\, d\mu(\xi)\notag\\
&=\int_{\K(X)}\int_{\K(\Lambda)}\sum_{x\in\tau(\eta)}f(s_x,x)g(\eta)\, \frac{1}{Z_{\Lambda}(\xi)} e^{-H_{\Lambda}(\eta\mid \xi_{\Lambda^c})}d\nu_{\Lambda}(\eta)\, d\mu(\xi)\notag\\
&=\int_{\K(X)}\int_{\K(\Lambda)}\int_{\Rp\times\Lambda}f(s,x)g(\eta+s\delta_{x})\, \frac{1}{Z_{\Lambda}(\xi)} e^{-H_{\Lambda}(\eta+s\delta_{x}\vert \xi_{\Lambda^c})}\, d\sigma(s,x)d\nu_{\Lambda}(\eta)\, d\mu(\xi)\notag\\
&=\int_{\hat X}\int_{\K(X)}\int_{\K(X)}F(s,x,\eta+s\delta_{x}) \exp\left[-s\int_{X\setminus\{x\}}\phi(x,x')\,d\eta(x')\right] \,  \pi_{\Lambda}(d\eta\mid\xi)\, d\mu(\xi)\,d\sigma(s,x)\notag\\
&=\int_{\K(X)}\int_{\hat X}\exp\left[-s\int_{X\setminus\{x\}}\phi(x,x')\,d\eta(x')\right]F(s,x,\eta+s\delta_x) \,d\sigma(s,x) d\mu(\eta),\label{utfty7roy}
\end{align}
where the last line is obtained by applying the DLR equation (\ref{DLR}) again. Note that, for a fixed $\eta\in\K(X)$, since the set $\tau(\eta)$ is countable, we have $\sigma(\tau(\eta)\times\R_+)=0$. Hence, in formula \eqref{utfty7roy}, instead of the integral $\int_{X\setminus\{x\}}\phi(x,x')\,d\eta(x')$, we may write $\int_{X}\phi(x,x')\,d\eta(x')$.
\end{proof}

\section{The results}\label{hugfu7tr}

In this section, we will introduce the Dirichlet form $\mathscr E^\K$ and formulate the results. We postpone the proofs to Section~\ref{ft67r67}.

Let the conditions of Proposition \ref{hguf7tuo} be satisfied and let us fix any Gibbs measure $\mu\in G(\nu,\phi)$. For any $F,G\in\FCK$, we define $\mathscr E^{\K}(F,G)$ by formula \eqref{hfyrf7}. Note that, by \eqref{guft} and \eqref{iyutfr76r}, we indeed have
$$ \int_{\K(X)}\big|\la\nabla^\K F(\eta), \nabla^\K G(\eta)\big\ra_{T_\eta(\K)}\big|\,d\mu(\eta)<\infty. $$

\begin{lemma}\label{bhuyft7}
Let $F,G\in\FCK$ and let $F=0$ $\mu$-a.e. Then $\mathscr E^{\K}(F,G)=0$.
\end{lemma}

Thus, we may consider $\mathscr E^\K$ as a symmetric bilinear form on $L^2(\K(X),\mu)$ with domain $\FCK$. Note that $\FCK$ is dense in $L^2(\K(X),\mu)$.
Let us now find the $L^2$-generator of this form. Analogously to \eqref{gfuyfr}, \eqref{vyd6trde}, we define, for each function $F\in\FCK$,
$\eta\in\K(X)$, and $x\in\tau(\eta)$,
\begin{align*}
\Delta_xF(\eta):=&\Delta_y\big|_{y=x}F(\eta-s_x\delta_x+s_x\delta_y),\\
\Delta_{s_x}F(\eta):=&\frac{d^2}{du^2}\Big|_{u=s_x}F(\eta-s_x\delta_x+u\delta_x),
\end{align*}
where $\Delta_y$ is the Laplace operator on $X$ acting in the $y$ variable.

The following proposition gives, in particular, the explicit form of the $L^2$-generator of the bilinear form $(\mathscr E^\K,\FCK)$.

\begin{proposition}\label{jig8u} Assume that $l\in C^1(\hat X)$ and $\phi\in C^1(X\times X)$. For each $F\in\nolinebreak \FCK$, we define a function $L^\K F\in L^2(\K(X),\mu)$ by
\begin{align}
L^\K F(\eta)&=\sum_{x\in\tau(\eta)}\bigg[\frac1{s_x}\Delta_x F(\eta)+\frac1{s_x}\la \nabla_x\log l(s,x),\nabla_x F(\eta)\ra_X\notag\\
&\quad -\int_X d(\eta-s_x\delta_x)(x')\la \nabla_x\phi(x,x'),\nabla_x F(\eta)\ra_X\notag\\
&\quad+ s_x\Delta_{s_x}F(\eta)+s_x\big(\nabla_{s_x}\log l(s_x,x)\big)\big(\nabla_{s_x}F(\eta)\big)\notag\\
&\quad-\left(\int_X d(\eta-s_x\delta_x)(x') \phi(x,x')\right) s_x\nabla_{s_x}F(\eta)\bigg].\label{huf7tur}
\end{align}
Here $\la\cdot,\cdot\ra_X$ denotes the scalar product in $X$.
Then, for any $F,G\in\FCK$,
\begin{equation}\label{gyfdy6}
\mathscr E^\K(F,G)=(-L^\K F,G)_{L^2(\K(X),\,\mu)}.\end{equation}
The bilinear form $(\mathscr E^\K,\FCK)$ is closable on $L^2(\K(X),\mu)$, and its closure, denoted by $(\mathscr E^\K,D(\mathscr E^\K))$ is a Dirichlet form. The operator $(-L^\K,\FCK)$ has Fried\-richs' extension, which we denote by $(-L^\K,D(L^\K))$.
\end{proposition}

\begin{remark}
Note that, in the case where $\mu$ is the Gibbs perturbation of the gamma measure, i.e., when $l(s,x)=e^{-s}$, formula \eqref{huf7tur} becomes
\begin{align*}
L^\K F(\eta)&=\sum_{x\in\tau(\eta)}\bigg[\frac1{s_x}\Delta_x F(\eta)-\int_X d(\eta-s_x\delta_x)(x')\la \nabla_x\phi(x,x'),\nabla_x F(\eta)\ra_X\notag\\
&\quad+ s_x\big(\Delta_{s_x}F(\eta)-\nabla_{s_x}F(\eta)\big)
-\left(\int_X d(\eta-s_x\delta_x)(x') \phi(x,x')\right) s_x\nabla_{s_x}F(\eta)\bigg].
\end{align*}
\end{remark}

We are now ready to formulate the main result of the paper.

\begin{theorem}\label{vyfd7yr} Assume that the conditions of Propositions \ref{hguf7tuo} and~\ref{jig8u} be satisfied.
Further assume that the dimension $d$ of the space $X$ is $\ge2$.  Then
  there exists a conservative diffusion process on $\K(X)$
(i.e., a conservative strong Markov process with
continuous sample paths in $\K(X)$), $$M^\K =(\Omega^\K ,\mathscr
F^\K ,(\mathscr  F^\K _t)_{t\ge0},( \Theta^\K _t)_{t\ge0}, (\mathfrak X^\K (t))_{t\ge 0},(\mathbb P^\K  _\eta)_{\eta\in\K(X)}),$$
(cf.\ \cite{Dy65})
 which is properly
associated with the Dirichlet form $(\mathscr E^\K , D(\mathscr E^\K ))$, i.e., for all ($\mu$-versions
of) $F\in L^2(\K(X),\mu)$ and all $t>0$ the function
$$\K(X)\ni\eta\mapsto
(p^\K_tF)(\eta){:=}\int_{\Omega } F(\mathfrak X (t))\, d\mathbb P^\K _\eta$$
 is an $\mathscr E^\K$-quasi-continuous version of
$\exp(tL^\K)F$ (cf.\ \cite[Chap.~1, Sect.~2]{MR1}). Here $\Omega^\K =C([0,\infty)\to\K(X))$, $\mathfrak X ^\K (t)(\omega)=\omega(t)$, $t\ge 0$, $\omega\in\Omega^\K $,
$(\mathscr  F^\K _t)_{t\ge 0}$ together with $\mathscr F^\K $ is the corresponding
minimum completed admissible family (cf.\
\cite[Section~4.1]{Fu80})  and $\Theta^\K _t$, $t\ge0$, are the
corresponding natural time shifts.

In particular, $M^\K$ is $\mu$-symmetric
(i.e., $\int G\, p^\K _tF\, d\mu=\int F \, p^\K_t G\, d\mu$
for all $F,G:\K(X)\to[0,\infty)$, $\mathscr B
(\K(X))$-measurable) and has $\mu$ as an invariant
measure.

$M^\K$ is up to $\mu$-equivalence unique (cf.\ \cite[Chap.~IV, Sect.~6]{MR1}).
\end{theorem}

\begin{remark} In addition to \eqref{hfjiu}--\eqref{khgiygt}, let us assume that
 the function $l(s,x)$ satisfies, for each $\Lambda\in\mathscr B_0(X)$,
$$ \int_{\Rp\times\Lambda}l(s,x)s^i\,ds\,dx<\infty,\quad i=2,3.$$
 This implies that the completely random measure $\nu$ satisfies, for each $\Lambda\in\mathscr B_0(X)$,
$$\int_{\K(X)}\eta(\Lambda)^n\,d\nu(\eta)<\infty\quad \text{for }n=1,2,3,4.$$
Then it easily follows from the proofs of
 Proposition \ref{jig8u} and Theorem \ref{vyfd7yr} that  these statements remain true when $l\in C^1(\hat X)$ and the pair potential $\phi$ is equal to zero, i.e., when $\mu=\nu$.

 We note that, in paper \cite{KLV}, for a different choice of a tangent space $T_\eta(\K)$ and   in the case where $l(s,x)=l(s)$ is independent of $x$ and $\mu=\nu$, the corresponding diffusion process on $\K(X)$ was constructed explicitly.
However, for the choice of the tangent space $T_\eta(\K)$ as in this paper, even in the case where $\mu=\nu$, an explicit construction of the diffusion process is an open problem, see Subsec.~5.2 in \cite{KLV}.
\end{remark}

\section{The proofs}\label{ft67r67}

\subsection{Proofs of Lemma \ref{bhuyft7} and Proposition \ref{jig8u}}

We start with the following

\begin{lemma}\label{gfdtst6r}
For any $F,G\in\FCK$,
\begin{align}
&\mathscr E^\K(F,G)=\int_{\K(X)}d\mu(\eta)\int_{\hat X}ds\,dx\, l(s,x)\,\exp\left[-s\int_X \phi(x,x')\,d\eta(x')\right]\notag\\
&\quad\times\bigg[\frac1{s^2}\la \nabla_x F(\eta+s\delta_x),\nabla_xG(\eta+s\delta_x)\ra_X +\bigg(\frac{d}{ds}F(\eta+s\delta_x)\bigg)\bigg(\frac{d}{ds}G(\eta+s\delta_x)\bigg)\bigg].\label{gfytfdy7}
\end{align}
\end{lemma}

\begin{proof} Formula \eqref{gfytfdy7} follows directly from \eqref{hfyrf7}, \eqref{hjfuy8y}, \eqref{gfuyfr}--\eqref{guft}, and \eqref{ghydcydr}.
\end{proof}

\begin{proof}[Proof of Lemma \ref{bhuyft7}] By (C1) and \eqref{iyutfr76r}, for a fixed $x\in X$, we get
$$\int_{\K(X)} \int_X |\phi(x,x')|\,d\eta(x')\,d\mu(\eta)<\infty.$$
Hence, for $\mu$-a.a.\ $\eta\in\K(X)$,  we have $  \int_X |\phi(x,x')|\,d\eta(x')<\infty$. Therefore,
 on $\hat X\times\K(X)$, the measures
$$l(s,x)\exp\bigg[-s\int_X \phi(x,x')\,d\eta(x')\bigg]ds\,dx \,d\mu(\eta)$$
and $ds\,dx \,d\mu(\eta)$ are equivalent.

Let $F\in\FCK$ be such that  $F=0$ $\mu$-a.e. Then, for any $\Lambda\in\mathscr B_0(X)$, we get by \eqref{ghydcydr}
\begin{align*}&\int_{\K(X)}d\mu(\eta)\int_{\hat X} ds\,dx\, l(s,x)\exp\bigg[-s\int_X \phi(x,x')\,d\eta(x')\bigg] |F(\eta+s\delta_x)|\chi_\Lambda(x)\\
&\quad=\int_{\K(X)} |F(\eta)|\, \eta(\Lambda)\, d\mu(\eta)=0.\end{align*}
Here $\chi_\Lambda$ denotes the indicator function of the set $\Lambda$.
Hence, $F(\eta+s\delta_x)=0$ for $ds\,dx \,d\mu(\eta)$-a.a.\ $(s,x,\eta)\in \hat X\times\K(X)$. For each fixed $\eta\in\K(X)$, the function $(s,x)\mapsto F(\eta+s\delta_x)$ is continuous. Therefore, for $\mu$-a.a.\ $\eta\in\K(X)$, $F(\eta+s\delta_x)=0$ for all $(s,x)\in \hat X$. Hence, by Lemma~\ref{gfdtst6r}, for each $G\in\FCK$, $\mathscr E^\K(F,G)=0$.
\end{proof}

\begin{proof}[Proof of  Proposition \ref{jig8u}]
 We first note that $(\mathscr E^\K,\FCK)$ is a pre-Dirichlet form form on $L^2(\K(X),\mu)$, i.e., if it is closable then its closure is a Dirichlet form. This assertion follows, by standard methods, directly from \cite[Chap.~I, Proposition 4.10]{MR1} (see also \cite[Chap.~II, Exercise 2.7]{MR1}).

 For a fixed $\eta\in\K(X)$, the function $(s,x)\mapsto F(\eta+s\delta_x)$ is constant outside a compact set in $\hat X$. Note also that, for each fixed $\eta\in\K(X)$, the function $x\mapsto\int_X \phi(x,x')\,d\eta(x')$ is differentiable on $X$ and its gradient is equal to $\int_X \nabla_x\phi(x,x')\,d\eta(x')$.
Hence carrying out integration by parts in formula \eqref{gfytfdy7}, we get for any $F,G\in\FCK$,
\begin{align*}
&\mathscr E^\K(F,G)=\int_{\K(X)}d\mu(\eta)\int_{\hat X}ds\,dx\, l(s,x)\,\exp\left[-s\int_X \phi(x,x')\,d\eta(x')\right]
G(\eta+s\delta_x)\\
&\quad\times\bigg[
-\frac1{s^2}\Delta_x F(\eta+s\delta_x)-\frac1{s^2}\la \nabla_x \log l(s,x),\nabla_x F(\eta+s\delta_x)\ra_X\\
&\quad + \frac1s  \int_{X} d\eta(x')\,\big\la\nabla_x\phi(x,x'),\nabla_x F(\eta+s\delta_x)\big\ra_X-\Delta_s F(\eta+s\delta_x)\\
&\quad -
\big(\nabla_s \log l(s,x)\big)\big(\nabla_s F(\eta+s\delta_x)\big) +\bigg(\int_X \phi(x,x')\,d\eta(x')\bigg)\big(\nabla_s F(\eta+s\delta_x)\big)
\bigg].
\end{align*}
Applying formula \eqref{ghydcydr}, we get \eqref{huf7tur}, \eqref{gyfdy6}.

It  easily follows from \eqref{huf7tur} that, for a fixed $F\in\FCK$, there exist $\Lambda\in\mathscr B_0(X)$ and $C>0$ such that
$$ |L^\K F(\eta)|\le C(\eta(\Lambda)+\eta(\Lambda)^2),\quad \eta\in\K(X). $$
Hence, by \eqref{iyutfr76r}, $L^\K F\in L^2(\K(X),\mu)$. Thus, the bilinear form $(\mathscr E^\K,\FCK)$ has $L^2$-generator. Hence, it is closable and its closure is a Dirichlet form. The last statement of the proposition about Friedrichs' extension is a standard fact of functional analysis.
\end{proof}

\subsection{Proof of Theorem \ref{vyfd7yr}}

We will divide the proof into several steps.

{\it Step 1.} To prove the theorem, we will initially construct a diffusion process on a certain subset of the configuration space over $\hat X$.  So in this step, we will  present the necessary definitions and constructions related to the configuration space.

We denote by $\ddot\Gamma(\hat X)$ the space of all $\N_0\cup\{\infty\}$-valued Radon measures on $\hat X$.
 Here $\N_0:=\{0,1,2,\dots\}$.
 The space $\ddot\Gamma(\hat X)$ is endowed  with the vague topology and let $\mathscr B(\ddot\Gamma(\hat X))$ denote the corresponding $\sigma$-algebra.

The configuration space over $\hat X$, denoted by $\Gamma(\hat X)$, is defined as the collection of all locally finite subsets of $\hat X$:
$$\Gamma(\hat X):=\big\{\gamma\subset \hat X\mid |\gamma\cap A|<\infty\text{ for each compact }A\subset \hat X\,\big\}.
$$
Here $|\gamma\cap A|$ denotes the cardinality of the  set $\gamma\cap A$.
One usually identifies a configuration $\gamma\in\Gamma(\hat X)$ with the Radon measure $\sum_{(s,x)\in\gamma}\delta_{(s,x)}$ on $\hat X$. Thus, one gets the inclusion $\Gamma(\hat X)\subset \ddot\Gamma(\hat X)$.

Let $\Gamma_{pf}(\hat X)$ denote the  subset of $\Gamma(\hat X)$ which consists of all configurations $\gamma$ which satisfy:

\begin{itemize}
\item[(i)] if $(s_1,x_1),(s_2,x_2)\in\gamma$ and $(s_1,x_1)\ne(s_2,x_2)$, then $x_1\ne x_2$;

\item[(ii)] for each  $\Lambda\in\mathscr B_0(X)$, $\displaystyle\sum_{(s,x)\in\gamma\cap (\Rp\times \Lambda)}s<\infty$.
\end{itemize}

We have $\Gamma_{pf}(\hat X)\in\mathscr B(\ddot\Gamma(\hat X))$, and we
denote by $\mathscr B(\Gamma_{pf}(\hat X))$ the trace $\sigma$-algebra of $\mathscr B(\ddot\Gamma(\hat X))$  on $\Gamma_{pf}(\hat X)$. Equivalently, $\mathscr B(\Gamma_{pf}(\hat X))$ is the Borel $\sigma$-algebra on the space $\Gamma_{pf}(\hat X)$ equipped with the vague topology.

The following statement is proven in \cite[Theorem~6.2]{HKPR}.

\begin{proposition}[\cite{HKPR}]\label{ghig}
Consider a bijective mapping
$\mathscr R  :\Gamma_{pf}(\hat X)\to\K(X) $ defined by
\begin{equation}\label{igyu8}
\Gamma_{pf}(\hat X)\ni \gamma=\{(s_i,x_i)\}\mapsto \mathscr R\gamma:=\sum_i s_i\delta_{x_i}\in \K(X).\end{equation}
Then  the mapping $\mathscr R$ and its inverse $\mathscr R^{-1}:\K(X)\to \Gamma_{pf}(\hat X)$ are measurable.
\end{proposition}

Note that the pushforward of the completely random measure $\nu$ under $\mathscr R^{-1}$ is the Poisson measure on $\Gamma(\hat X)$ with intensity measure $\sigma$: if we denote this measure by $\pi$, the Fourier transform of $\pi$ is given by
$$ \int_{\Gamma_{pf}(\hat X)}e^{i\la f,\gamma\ra}\,d\pi(\gamma)=\exp\bigg[\int_{\hat X}(e^{if(s,x)}-1)\,d\sigma(s,x)\bigg],\quad f\in C_0(\hat X).$$
Here we denote $\la f,\gamma\ra:=\int_{\hat X}f\,d\gamma=\sum_{(s,x)\in\gamma}f(s,x)$.

Let $\rho$ denote the pushforward of the Gibbs measure $\mu$ under $\mathscr R^{-1}$. By Theorem~\ref{hjugfuyt8} and \eqref{igyu8}, the measure $\rho$ satisfies,
for each measurable function $F:\hat X\times\Gamma(\hat X)\to[0,\infty]$,
\begin{multline*}
\int_{\Gamma_{pf}(\hat X)}\sum_{(s,x)\in\gamma}F(s,x,\gamma)\,d\rho(\gamma)\\
=\int_{\Gamma_{pf}(\hat X)}d\rho(\gamma)\int_{\hat X}d\sigma(s,x)\,\exp\left[-\sum_{(s',x')\in\gamma}ss'\phi(x,x')
\right]F(s,x,\gamma\cup \{(s,x)\}).\end{multline*}

 Let $\FCG$ denote the set of functions on $\Gamma_{pf}(\hat X)$ which are of the form $F(\gamma)=G(\mathscr R\gamma)$ for some $G\in\FCK$. Thus, $\FCG$ consists of all functions $F$ of the form
$$
 F(\gamma)=g\big(\langle\varphi_1,\gamma\rangle,\dots,\langle \varphi_N,\gamma\rangle\big),\quad \gamma\in\Gamma_{pf}(\hat X),$$
where the functions $g,\varphi_1,\dots,\varphi_N$ are as in \eqref{hufutr}.  Thus, we may equivalently consider a bilinear form $(\mathscr E^\Gamma,\FCG)$ on $L^2(\Gamma_{pf}(\hat X),\rho)$ which is defined by
$$\mathscr E^\Gamma(F,G):=\mathscr E^\K(F\circ\mathscr R^{-1},G\circ\mathscr R^{-1}),\quad F,G\in\FCG.$$
As easily seen, for any $F,G\in\FCG$, we have
$$\mathscr E^\Gamma (F,G)=\int_{\Gamma(\hat X)}\sum_{(s,x)\in\gamma}\bigg[
\frac1s\la\nabla_x F(\gamma),\nabla_x G(\gamma)\ra_X+s\big(\nabla_s F(\gamma)\big(\nabla_s G(\gamma)\big)
\bigg] d\rho(\gamma),$$
where $\nabla_x F(\gamma)$ and $\nabla_s G(\gamma)$
are defined analogously to formulas \eqref{gfuyfr}, \eqref{vyd6trde}.
By Proposition~\ref{jig8u}, the bilinear form $(\mathscr E^\Gamma,\FCG)$ is closable on $L^2(\Gamma_{pf}(\hat X),\rho)$, and its closure, denoted by $(\mathscr E^\Gamma,D(\mathscr E^\Gamma))$, is a Dirichlet form.

{\it Step 2.}  Our aim now is to construct a diffusion process on $\Gamma_{pf}(\hat X)$ which is properly associated with the Dirichlet form
 $(\mathscr E^\Gamma,D(\mathscr E^\Gamma))$. We will initially construct such a process on a bigger space $\ddot\Gamma_f(\hat X)$. In this step, we will define the set $\ddot\Gamma_f(\hat X)$ and construct a metric on it such that the  set $\ddot\Gamma_f(\hat X)$ equipped with this metric is a Polish space.

  For each $\Lambda\in\mathscr B_0(X)$, we
  define a local mass $\mathfrak M_\Lambda$ by
$$\mathfrak M_\Lambda(\gamma):= \int_{\hat X}\chi_\Lambda(x)s\,d\gamma(s,x),\quad \gamma\in\ddot\Gamma(\hat X).$$
 We set
$$\ddot\Gamma_{f}(\hat X):=\big\{\gamma\in \ddot\Gamma(\hat X)\mid \mathfrak M_\Lambda(\gamma)<\infty\text{ for each  }\Lambda\in\mathscr B_0(X)\big\}.$$
We have $\ddot\Gamma_{f}(\hat X)\in\mathscr B(\ddot\Gamma(\hat X))$, and let $\mathscr B(\ddot\Gamma_{f}(\hat X))$ denote the Borel $\sigma$-algebra on the space $\ddot\Gamma_{f}(\hat X)$ equipped with the vague topology.

We will now construct a bounded metric on $\ddot\Gamma_{f}(\hat X)$ in which this space will be complete and separable.  Let $d_V(\cdot,\cdot)$ denote the bounded metric on $\ddot\Gamma(\hat X)$ which was introduced in \cite[Section~3]{MR2}. Recall that this metric generates the vague topology on $\ddot\Gamma(\hat X)$, and $\ddot\Gamma(\hat X)$ is complete and separable in this metric.

For each $k\in\N$, we fix any function $\phi_k\in C_0^\infty(X)$ such that
\begin{align}&\chi_{B(k)}\le\phi_k\le\chi_{B(k+1)},\quad
\left|\frac{\partial}{\partial x_i}\,\phi_k(x)\right|\le2\,\chi_{B(k+1)}(x),\notag\\
&\qquad  i=1,\dots,d,\ x=(x^1,\dots,x^d)\in X. \label{1}
\end{align}
 Here
 $$B(k):=
\big\{x=(x^1,\dots,x^d)\in X\mid \max_{i=1,\dots,d}|x_i|\le k\big\}.$$
 Next, we fix any $q\in(0,1)$. We take any  sequence $(\psi_n)_{n\in\Z}$ such that, for each $n\in\Z$, $\psi_n\in C_0^\infty(\R)$  and
\begin{equation}\label{yufr8r}
\chi_{[q^n,\,q^{n-1}]}\le \psi_n\le \chi_{[q^{n+1},\,q^{n-2}]},\quad |\psi_n'|\le \frac2{q^n-q^{n+1}}\,\chi_{[q^{n+1},\,q^{n}]\cup[q^{n-1},\,q^{n-2}]}.
\end{equation}
For each $k\in\N$ and $n\in\Z$, we define
\begin{equation}\label{hgu8ytfk}\varkappa_{kn}(s,x):=\phi_k(x)\psi_n(s)s,\quad (s,x)\in\hat X.\end{equation}
 Note that $\varkappa_{kn}\in C_0^\infty(\hat X)$. For any
 $k\in\N$ and
 $\gamma,\gamma'\in\ddot\Gamma_{f}(\hat X)$, we define
\begin{equation}\label{rererre}
d_k(\gamma,\gamma'):=\sum_{n\in\Z}|\la \varkappa_{kn},\gamma-\gamma'\ra|.\end{equation}
As follows from \eqref{1} and \eqref{yufr8r}, for each $\gamma\in\ddot\Gamma_f(\hat X)$,
\begin{align}
\sum_{n\in\Z}\la \varkappa_{kn},\gamma\ra&=\int_{\hat X}d\gamma(s,x)
\,\phi_k(x)\left(\sum_{n\in\Z}\psi_n(s)\right)s\notag\\
&\le 4\int_{\hat X}d\gamma(s,x)\phi_k(x)s\le 4\,\mathfrak M_{B(k+1)}(\gamma)<\infty.\label{bhitgiftr}\end{align}
 Therefore, $d_k(\gamma,\gamma')<\infty$ for all $\gamma,\gamma'\in\ddot\Gamma_{f}(\hat X)$. Clearly,  $d_k(\cdot,\cdot)$ satisfies the triangle inequality.

 Let $(c_k)_{k=1}^\infty$ be a sequence of $c_k>0$ such that $\sum_{k=1}^\infty c_k<\infty$. Below, in formula~\eqref{yut8}, we will make an explicit choice of the sequence $(c_k)_{k=1}^\infty$.
  We next define
$$ d_f(\gamma,\gamma'):=\sum_{k=1}^\infty c_k\, \frac{d_k(\gamma,\gamma')}{1+d_k(\gamma,\gamma')}\,,\quad \gamma,\gamma'\in\ddot\Gamma_f(\hat X).$$
Clearly, $d_f(\cdot,\cdot)$ also satisfies the triangle inequality.
We finally define the metric
$$d(\gamma,\gamma'):=d_V(\gamma,\gamma')+d_f(\gamma,\gamma'),\quad \gamma,\gamma'\in\ddot\Gamma_f(\hat X).$$

\begin{proposition}\label{jkgiugt}
$(\ddot\Gamma_f(\hat X), d(\cdot,\cdot))$ is a complete, separable metric space.
\end{proposition}

\begin{proof}
Let $\{\gamma_i\}_{i=1}^\infty$ be a Cauchy sequence in $(\ddot\Gamma_f(\hat X), d(\cdot,\cdot))$. Then $\{\gamma_i\}_{i=1}^\infty$ is a Cauchy sequence in $(\ddot\Gamma(\hat X),d_V(\cdot,\cdot))$. Since the latter space is complete, there exists $\gamma\in \ddot\Gamma(\hat X)$ such that $\gamma_i\to\gamma$ vaguely as $i\to\infty$.
Denote
$$ a_{kn}^{(i)}:=\la \varkappa_{kn},\gamma_i\ra,\quad a_{kn}:=\la \varkappa_{kn},\gamma\ra,\quad k\in\N,\ n\in\Z.$$
As $\varkappa_{kn}\in C_0(\hat X)$, we therefore get:
\begin{equation}\label{mnohy} \text{for each $k\in\N$ and $n\in\Z$}\quad a_{kn}^{(i)}\to a_{kn}\text{ as }i\to\infty.\end{equation}
Note that, for each $k\in\N$ and $i\in\N$, $a_{kn}^{(i)}\ge0$ for all $n\in\Z$ and by \eqref{bhitgiftr}
$$\sum_{n\in\N}a_{kn}^{(i)}<\infty.$$
Hence, $(a_{kn}^{(i)})_{n\in\Z}\in\ell^1(\Z)$.
As $\{\gamma_i\}_{i=1}^\infty$ is a Cauchy sequence in $(\ddot\Gamma_f(\hat X), d(\cdot,\cdot))$,
$$ \lim_{i,j\to\infty}\sum_{n\in\Z}|a_{kn}^{(i)}-a_{kn}^{(j)}|=\lim_{i,j\to\infty} d_k(\gamma_i,\gamma_j)=0,\quad k\in\N.$$
Hence, $\{(a_{kn}^{(i)})_{n\in\Z}\}_{i=1}^\infty$ is a Cauchy sequence in $\ell^1(\Z)$. Since the latter space is complete, the sequence $\{(a_{kn}^{(i)})_{n\in\Z}\}_{i=1}^\infty$ is convergent in $\ell^1(\Z)$. In view of \eqref{mnohy}, we therefore conclude that the $\ell^1(\Z)$-limit of this sequence is $(a_{kn})_{n\in\Z}$. This, in particular, implies that
\begin{equation}\label{vgyudr7ky}\sum_{n\in\Z}a_{kn}=\sum_{n\in\Z}\la \varkappa_{kn},\gamma\ra<\infty,\quad k\in\N.\end{equation}
By \eqref{yufr8r}, $\sum_{n=1}^\infty \psi_n(s)\ge 1$ for all $s\in\Rp$.  We therefore deduce from \eqref{vgyudr7ky} that $\gamma\in\ddot\Gamma_f(\hat X)$. Furthermore,
$$ d_k(\gamma_i,\gamma)=\sum_{n\in\Z}|a_{kn}^{(i)}-a_{kn}|\to0\quad\text{as }i\to\infty,\quad k\in\N.$$
 Hence $d(\gamma_i,\gamma)\to0$ as $i\to\infty$. Thus,
$(\ddot\Gamma_f(\hat X), d(\cdot,\cdot))$ is  complete. The proof of the separability of this space is routine, so we skip it.
\end{proof}

{\it Step 3}. We will now consider $(\mathscr E^\Gamma,D(\mathscr E^\Gamma))$ as a Dirichlet form on $L^2(\ddot\Gamma_f(\hat X)),\rho)$ and prove that is is quasi-regular. For the definition of quasi-regularity of a Dirichlet form,  see \cite[Chap.~IV, Def.~3.1]{MR1} and \cite[subsec.~4.1]{MR2}.

We  consider the complete separable metric space $(\ddot\Gamma_f(\hat X), d(\cdot,\cdot))$, and let $\mathscr B(\ddot\Gamma_f(\hat X),d)$ denote the corresponding Borel $\sigma$-algebra on $\ddot\Gamma_f(\hat X)$.

\begin{lemma}\label{bugyfyfy}
 We have $\mathscr B(\ddot\Gamma_{f}(\hat X))=\mathscr B(\ddot\Gamma_{f}(\hat X),d)$.
\end{lemma}

\begin{proof} We have $d(\gamma,\gamma')\ge d_V(\gamma,\gamma')$ for all $\gamma,\gamma'\in\ddot\Gamma_{f}(\hat X)$. Therefore, $\mathscr B(\ddot\Gamma_{f}(\hat X))\subset\mathscr B(\ddot\Gamma_{f}(\hat X),d)$. On the other hand, it follows from  the construction of the metric $d(\cdot,\cdot)$ that,  for a fixed $\gamma'\in\ddot\Gamma_f(\hat X)$, the function
$$\ddot\Gamma_f(\hat X)\ni\gamma\mapsto d(\gamma,\gamma')\in\R$$
is $\mathscr B(\ddot\Gamma_{f}(\hat X))$-measurable. Hence, for any $\gamma'\in\ddot\Gamma_f(\hat X)$ and $r>0$, 
\begin{equation}\label{khguygfu}
\{\gamma\in\ddot\Gamma_f(\hat X)\mid d(\gamma,\gamma')<r\}\in \mathscr B(\ddot\Gamma_{f}(\hat X)).\end{equation}
But in a  separable metric space, every open set can be represented as a countable union of open balls, see e.g.\ Theorem 2 and its proof in \cite[p.~206]{Kur}. Hence, \eqref{khguygfu} implies
 the inclusion $\mathscr B(\ddot\Gamma_{f}(\hat X),d)\subset\mathscr B(\ddot\Gamma_{f}(\hat X))$.
\end{proof}

We will now  consider 
$\rho$ as a probability measure on the measurable space $(\ddot\Gamma_f(\hat X),\linebreak\mathscr B(\ddot\Gamma_f(\hat X)))$, and $(\mathscr E^\Gamma,D(\mathscr E^\Gamma))$ as a Dirichlet form on the space $L^2(\ddot\Gamma_f(\hat X),\rho)$.

On $D(\mathscr E^\Gamma)$ we consider the norm
$$\|F\|_{D(\mathscr E^\Gamma)}:=\mathscr E^\Gamma(F,F)^{1/2}+\|F\|_{L^2(\ddot\Gamma_{f}(\hat X),\,\rho)}.$$
We define a square field operator
\begin{equation}\label{huf7u}
S^\Gamma(F)(\gamma):=\sum_{(s,x)\in\gamma}\bigg[
\frac1s\, \|\nabla_x F(\gamma)\|_X^2 +s\big|\nabla_s F(\gamma)\big|^2
\bigg], \end{equation}
where $F\in\FCG$,  $\gamma\in\Gamma_{pf}(\hat X)$, and $\|\cdot\|_X$ denotes the Euclidean norm in $X$. As easily seen, $S^\Gamma$ extends by continuity in the norm $\|\cdot\|_{D(\mathscr E^\Gamma)}$ to a mapping $S^\Gamma:D(\mathscr E^\Gamma)\to L^1(\ddot\Gamma_{f}(\hat X),\rho)$, and furthermore
$\mathscr E^\Gamma(F,F)=\int_{\ddot\Gamma_{f}(\hat X)} S^\Gamma(F)\,d\rho$.

\begin{lemma}\label{jkighiy8t} For each $\gamma\in\ddot\Gamma_f(\hat X)$, we have $d(\cdot,\gamma)\in D(\mathscr E^\Gamma)$. Furthermore, there exists $G\in L^1(\ddot\Gamma_{f}(\hat X),\rho)$ (independent of $\gamma$) such that $S^\Gamma(d(\cdot,\gamma))\le G$ $\rho$-a.e.
\end{lemma}

\begin{proof} Recall that $d(\cdot,\gamma)=d_V(\cdot,\gamma)+d_f(\cdot,\gamma)$. 
 Using the methods of \cite[Section~4]{MR2} (see also \cite[Section~6]{KLR}), one can  show that $d_V(\cdot,\gamma)\in D(\mathscr E^\Gamma)$ and there exists $G_1\in L^1(\ddot\Gamma_{f}(\hat X), \rho)$  (independent of $\gamma$) such that
$S^\Gamma(d_V(\cdot,\gamma))\le G_1$ $\rho$-a.e. Hence, we only need to prove that $d_f(\cdot,\gamma)\in D(\mathscr E^\Gamma)$ and there exists $G_2\in L^1(\ddot\Gamma_{f}(\hat X), \rho)$ (independent of $\gamma$) such that
$S^\Gamma(d_f(\cdot,\gamma))\le G_2$ $\rho$-a.e.

Analogously to the proof of \cite[Lemma 4.7]{MR2}, we fix any sequence $(\zeta_n)_{n=1}^\infty$ such that $\zeta_n\in C_0^\infty(\R)$, $\int_{\R}\zeta_n(t)\,dt=1$, $\zeta_n(t)=\zeta_n(-t)$ for all $t\in\R$, $\operatorname{supp}(\zeta_n)\subset(-1/n,1/n)$. We  define
$$ u_n(t):=\int_{\R}|t-t'|\zeta_n(t')\,dt'-\int_\R|t'|\zeta_n(t')\,dt',\quad t\in\R.$$
It is easy to check that, for each $n\in\N$, $u_n\in C^\infty(\R)$,  $|u_n(t)|\le|t|$, $u_n(t)\to|t|$ as $n\to\infty$ for each $t\in\R$, $u'_n(t)\to\operatorname{sign}(t)$ as $n\to\infty$ for each $t\in\R\setminus\{0\}$, and $|u_n'(t)|\le2$ for all
$t\in\R$.

Recall \eqref{hgu8ytfk} and \eqref{rererre}. For each $N\in\N$, we define
\begin{align}&d_k^{(N)}(\gamma,\gamma'):=\sum_{n\in\Z\cap[-N,N]}u_N(\la\varkappa_{kn},\gamma-\gamma'\ra),\notag\\
&d^{(N)}_f (\gamma,\gamma'):= \sum_{k=1}^N  c_k\, \frac{d^{(N)}_k(\gamma,\gamma')}{1+d^{(N)}_k(\gamma,\gamma')},
\quad \gamma,\gamma'\in\ddot\Gamma_f(\hat X). \label{jftu}
\end{align}
Clearly, for a fixed $\gamma'\in\ddot\Gamma_f(\hat X)$,  the restriction of $d^{(N)}_f(\cdot,\gamma')$ to $\Gamma_{pf}(\hat X)$ belongs to $\FCG$. Hence, $d^{(N)}_f(\cdot,\gamma')\in D(\mathscr E^\Gamma)$.

 As easily seen, for each $\gamma\in\ddot\Gamma_{f}(\hat X)$, we have $d^{(N)}_f (\gamma,\gamma')\to d_f (\gamma,\gamma')$ as $N\to\infty$. Hence,
\begin{equation}\label{hyt97t}
d^{(N)}_f(\cdot,\gamma')\to d_f(\cdot,\gamma')\quad \text{in }L^2(\ddot\Gamma_{f}(\hat X),\rho)\text{ as }N\to\infty.\end{equation}

Note that, for $t\ge0$, $\left(\frac t{1+t}\right)'=\frac1{(1+t)^2}\le1$.
Hence, by \eqref{1}--\eqref{hgu8ytfk}, for each $\gamma\in\Gamma_{pf}(\hat X)$ and each $(s,x)\in\gamma$,
 \begin{align}
\|\nabla_x\, d_f^{(N)}(\gamma,\gamma')\|_X&\le \sum_{k=1}^N c_k\, \|\nabla_x\, d_k^{(N)}(\gamma,\gamma')\|_X \notag\\
&\le 2 \sum_{k=1}^N c_k \sum_{n\in\Z\cap[-N,N]}\|\nabla_x\,\varkappa_{kn}(x,s)\|_X \notag\\
&=2 \sum_{k=1}^N c_k\, \|\nabla\,\phi_k(x)\|_X\sum_{n\in\Z\cap[-N,N]}\psi_n(s)s \notag\\
&\le 4 \sqrt{d}\,\sum_{k=1}^\infty c_k\chi_{B(k+1)}(x)\sum_{n\in\Z\cap[-N,N]}\psi_n(s)s \notag\\
& \le
16 \sqrt{d}\,\sum_{k=1}^\infty c_k\chi_{B(k+1)}(x)s.\notag
\end{align}
Hence, using the Cauchy inequality, we conclude that there exists a constant $C_1>0$ such that
\begin{equation}\label{hyur7r}
\|\nabla_x\, d_f^{(N)}(\gamma,\gamma')\|^2_X\le C_1s^2\sum_{k=1}^\infty c_k\chi_{B(k+1)}(x).
\end{equation}

Analogously, using  \eqref{1}--\eqref{hgu8ytfk}, we get
\begin{align*}
&\big|\nabla_s d_f^{(N)}(\gamma,\gamma')\big|\le \sum_{k=1}^N c_k \big|\nabla_s d_k^{(N)}(\gamma,\gamma')\big|\\
&\quad\le 2 \sum_{k=1}^N c_k \sum_{n\in\Z\cap[-N,N]}\left|\frac{\partial}{\partial s}\,\varkappa_{kn}(x,s)\right|\\
&\quad=2 \sum_{k=1}^N c_k \phi_k(x)\sum_{n\in\Z\cap[-N,N]}|\psi'_n(s)s+\psi_n(s)|\\
&\quad \le 2 \sum_{k=1}^\infty c_k\chi_{B(k+1)}(x)\sum_{n\in\Z}\left( \frac2{q^n(1-q)}\chi_{[q^{n+1},\,q^n]\cup[q^{n-1},\,q^{n-2}]}(s)s+\chi_{[q^{n+1},\,q^{n-2}]}(s)\right)\\
&\quad \le 2 \sum_{k=1}^\infty c_k\chi_{B(k+1)}(x)\sum_{n\in\Z} \left(\frac2{q^n(1-q)}\chi_{[q^{n+1},\,q^n]\cup[q^{n-1},\,q^{n-2}]}(s)q^{n-2}+\chi_{[q^{n+1},\,q^{n-2}]}(s)\right)\\
&\quad \le 2 \sum_{k=1}^\infty c_k\chi_{B(k+1)}(x)\left(\frac8{q^2(1-q)}+4\right).
\end{align*}
Hence, there exists a constant $C_2>0$ such that
\begin{equation}\label{bhyt}
\big|\nabla_s F(\gamma)\big|^2\le C_2 \sum_{k=1}^\infty c_k\chi_{B(k+1)}(x).
\end{equation}

We define, for $\gamma\in\Gamma_{pf}(\hat X)$,
\begin{equation}\label{jkgiutrd}
G_2(\gamma):=(C_1+C_2)\sum_{(s,x)\in\gamma}s\sum_{k=1}^\infty c_k\chi_{B(k+1)}(x).\end{equation}
By the monotone convergence theorem,
\begin{align}
\int_{\ddot\Gamma_{f}(\hat X)} G_2\,d\rho&=(C_1+C_2)\sum_{k=1}^\infty c_k\int_{\Gamma_{pf}(\hat X)} \sum_{(s,x)\in\gamma}s\chi_{B(k+1)}(x)\,d\rho(\gamma)\notag\\
&=(C_1+C_2)\sum_{k=1}^\infty c_k\int_{\K(X)}\eta(B(k+1))\,d\mu(\eta).\label{hgfutt}
\end{align}
By \eqref{iyutfr76r}, we have, for each $k\in\N$,
$$ \int_{\K(X)}\eta(B(k+1))\,d\mu(\eta)<\infty.$$
So we may set
\begin{equation}\label{yut8}
 c_k:=2^{-k}\bigg(1+\int_{\K(X)}\eta(B(k+1))\,d\mu(\eta)\bigg)^{-1},\quad k\in\N.\end{equation}
Then, by \eqref{hgfutt}, we get $G_2\in L^1(\ddot\Gamma_{f}(\hat X,\rho))$. Furthermore, by \eqref{huf7u}, \eqref{hyur7r}--\eqref{jkgiutrd}, we get
\begin{equation}\label{igi}
 S^\Gamma(d_f^{(N)}(\cdot,\gamma'))\le G_2 \quad \text{point-wise on }\Gamma_{pf}(\hat X).\end{equation}

Using \eqref{igi} and the dominated convergence theorem, it is not hard to prove that
\begin{equation}\label{iglyt78p}\mathscr E^\Gamma\big(d_f^{(N)}(\cdot,\gamma')-d_f^{(M)}(\cdot,\gamma')\big)\to0\quad\text{as }N,M\to\infty.\end{equation}
Hence, $\big(d_f^{(N)}(\cdot,\gamma')\big)_{N=1}^\infty$ is a Cauchy sequence in
$(D(\mathscr E^\Gamma),\|\cdot\|_{D(\mathscr E^\Gamma)})$.  Hence, by \eqref{hyt97t} and \eqref{iglyt78p}, $d_f(\cdot,\gamma')\in D(\mathscr E^\Gamma)$. Furthermore, since $d_f^{(N)}(\cdot,\gamma')\to d_f(\cdot,\gamma')$ in the $\|\cdot\|_{D(\mathscr E^\Gamma)}$ norm,
$$S^\Gamma(d_f^{(N)}(\cdot,\gamma'))\to S^\Gamma(d_f(\cdot,\gamma'))\quad\text{in }L^1(\ddot\Gamma_{f}(\hat X),\rho)\text{ as }N\to\infty.$$
Hence, by \eqref{igi}, $S^\Gamma(d_f(\cdot,\gamma))\le G_2$ $\rho$-a.e.
\end{proof}

By \cite[Proposition~4.1]{MR2} (see also \cite[Theorem~3.4]{RS}), Proposition \ref{jkgiugt} and Lemma~\ref{jkighiy8t} imply the following proposition.

\begin{proposition}\label{gfuy} The Dirichlet form $(\mathscr E^\Gamma,D(\mathscr E^\Gamma))$ on $L^2(\ddot\Gamma_f(\hat X),\rho)$ is quasi-regular.
\end{proposition}

{\it Step 4.} We will now construct a corresponding diffusion process on $\ddot\Gamma_f(\hat X)$.

\begin{lemma}\label{hgiutg} The Dirichlet form  $(\mathscr E^\Gamma,D(\mathscr E^\Gamma))$ has local property, i.e., $\mathscr E^\Gamma(F,G)=0$ provided $F,G\in D(\mathscr E^\Gamma)$ with $\operatorname{supp}(|F|\rho)\cap \operatorname{supp}(|G|\rho)=\varnothing$.
\end{lemma}

\begin{proof}
Identical to the proof of \cite[Proposition~4.12]{MR2}.
\end{proof}

 As a consequence of Proposition \ref{gfuy}, Lemma \ref{hgiutg}, and
\cite[Chap.~IV, Theorem~3.5, and Chap.~V, Theorem~1.11]{MR1}, we
obtain

 \begin{proposition}\label{hjugfu7t}
There exists a conservative diffusion process on the metric space \linebreak $(\ddot\Gamma_f(\hat X),d(\cdot,\cdot))$,
 $$M^\Gamma =(\Omega^\Gamma ,\mathscr
F^\Gamma ,(\mathscr  F^\Gamma _t)_{t\ge0},( \Theta^\Gamma _t)_{t\ge0}, (\mathfrak X^\Gamma (t))_{t\ge 0},(\mathbb P^\Gamma  _\gamma)_{\gamma\in\ddot\Gamma_f(\hat X)}),$$
  which is properly associated with the Dirichlet form $(\mathscr E^\Gamma,D(\mathscr E^\Gamma))$.
   Here $\Omega^\Gamma=\linebreak C([0,\infty)\to\ddot\Gamma_f(\hat X))$, $\mathfrak X^\Gamma(t)(\omega)=\omega(t)$, $t\ge 0$, $\omega\in\Omega^\Gamma$,
$(\mathscr  F^\Gamma_t)_{t\ge 0}$ together with $\mathscr F^\Gamma$ is the corresponding
minimum completed admissible family, and $\Theta^\Gamma_t$, $t\ge0$, are the
corresponding natural time shifts.
  This process is up to $\rho$-equivalence unique.
\end{proposition}

{\it Step 5.} We will now show that the diffusion process from Proposition \ref{hjugfu7t} lives, in fact,  on the smaller space $\Gamma_{pf}(\hat X)$. This is where we use that the dimension $d$ of the underlying space $X$ is $\ge2$.

\begin{proposition}
The set $\ddot\Gamma_f(\hat X)\setminus \Gamma_{pf}(\hat X)$ is $\mathscr E^\Gamma$-exceptional. Thus, the statement of Proposition \ref{hjugfu7t} remains true if we replace in it $\ddot\Gamma_f(\hat X)$ with $\Gamma_{pf}(\hat X)$.
\end{proposition}

\begin{proof} The proof of this statement
is similar to the proof of \cite[Proposition~1 and Corollary 1]{RS2}, see also the proof of \cite[Theorem~6.3]{KLR}. \end{proof}

{\it Step 6.}  We will now prove that the mapping $\mathscr R$ is continuous with respect to the $d(\cdot,\cdot)$ metric.

\begin{proposition}\label{gsxgyaf} The mapping $\mathscr R$ acts continuously from the metric space $(\Gamma_{pf}(\hat X),d(\cdot,\cdot))$ into the space $\K(X)$ endowed with the vague topology.
\end{proposition}

\begin{proof}
 Let $\{\gamma_i\}_{i=1}^\infty\subset\Gamma_{pf}(\hat X)$ and $\gamma\in\Gamma_{pf}(\hat X)$. Let  $d(\gamma_i,\gamma)\to0$ as $i\to\infty$. We have to prove that $\mathscr R\gamma_i\to\mathscr R\gamma$ vaguely as $i\to\infty$.

So  fix any $f\in C_0(X)$ and  $\varepsilon>0$. Choose $k\in\N$ such that $\operatorname{supp}(f)\subset B(k)$. Choose $N\in\N$ such that
\begin{equation}\label{gyiur7oih}
\sum_{n\in\Z,\, |n|\ge N}\la \varkappa_{kn},\gamma\ra\le\varepsilon.
\end{equation}
Since $d(\gamma_i,\gamma)\to0$, we have $d_k(\gamma_i,\gamma)\to0$. Hence, there exists  $I\in\N$ such that
\begin{equation}\label{byufyr87}
\sum_{n\in\Z,\, |n|\ge N}\la\gamma_i,\varkappa_{kn}\ra\le2\varepsilon,\quad i\ge I.
\end{equation}
By \eqref{1}--\eqref{hgu8ytfk}, \eqref{gyiur7oih}, and \eqref{byufyr87},
\begin{align*}
&\int_{B(k)\times \left((0,\,q^N)\cup(q^{-N},\,\infty)\right)} s\,d\gamma(x,s)\le\varepsilon,\\
&\int_{B(k)\times \left((0,\,q^N)\cup(q^{-N},\,\infty)\right)} s\,d\gamma_i(x,s)\le2\varepsilon,\quad i\ge I.
\end{align*}
Therefore,
\begin{align}
&\int_{B(k)\times \left((0,\,q^N)\cup(q^{-N},\,\infty)\right)}|f(x)| s\,d\gamma(x,s)\le\varepsilon\|f\|_\infty,\notag\\
 &\int_{B(k)\times \left((0,\,q^N)\cup(q^{-N},\,\infty)\right)}|f(x)| s\,d\gamma_i(x,s)\le2\varepsilon\|f\|_\infty,\quad i\ge I,
\label{hg8t}
\end{align}
where $\|f\|_\infty$ is the supremum norm of the function $f$. Fix any $\xi\in  C_0(\Rp)$ such that
\begin{equation}\label{koh9t7ptp}\chi_{[q^{N},\,q^{-N}]}\le \xi\le1.\end{equation}
Since the function $f(x)\xi(s)s$ is from $C_0(\hat X)$, by the vague convergence
$$\int_{\hat X}f(x)\xi(s)s\,d\gamma_i(x,s)\to \int_{\hat X}f(x)\xi(s)s\,d\gamma(x,s)\quad\text{as }i\to\infty.$$
Hence, there exists $I_1\ge I$ such that
\begin{equation}\label{iyftr8rt} \bigg|\int_{\hat X}f(x)\xi(s)s\,d(\gamma_i-\gamma)(x,s)\bigg|\le \varepsilon,\quad i\ge I_1.\end{equation}
By \eqref{hg8t}--\eqref{iyftr8rt}, for all $i\ge I_1$,
\begin{align}
&\bigg|\int_{B(k)\times[q^{N},q^{-N}]}f(x)s\,d(\gamma_i-\gamma)(x,s)\bigg|=\bigg|\int_{B(k)\times[q^{N},q^{-N}]}f(x)\xi(s)s\,d(\gamma_i-\gamma)(x,s)\bigg|\notag\\
&\quad\le \bigg|\int_{\hat X}f(x)\xi(s)s\,d(\gamma_i-\gamma)(x,s)\bigg|\notag\\
&\qquad+ \bigg|
\int_{B(k)\times \left((0,\,q^N)\cup(q^{-N},
\,\infty)\right)}f(x)\xi(s)s\,d\gamma_i(x,s)
\bigg|\notag\\
&\qquad+\bigg|\int_{B(k)\times \left((0,\,q^N)\cup(q^{-N},\,\infty)\right)}f(x)\xi(s)s\,d\gamma(x,s)
\bigg|\notag\\
&\quad\le\varepsilon(1+3\|f\|_\infty).\label{lohoy}
\end{align}
By \eqref{hg8t} and \eqref{lohoy}, for all $i\ge I_1$,
$$\bigg|\int_X f(x)\,d(\mathscr R\gamma_i-\mathscr R\gamma)(x)\bigg|=\bigg|
\int_{\hat X}f(x)s\,d(\gamma_i-\gamma)(x,s)\bigg|\le \varepsilon(1+6\|f\|_\infty).  $$
Thus, the proposition is proven.
\end{proof}

{\it Step 7.} Finally, to construct the process $M^\K$ on $\K(X)$, we just map the
process $M^\Gamma$ from Proposition~\ref{hjugfu7t} onto $\K(X)$ by
 using the bijective mapping $\mathscr R :\Gamma_{pf}(\hat X)\to\mathbb K(X)$. Proposition~\ref{gsxgyaf}
ensures that the sample paths of the obtained Markov process are continuous in the vague topology on $\K(X)$.

\section*{Acknowledgements}
The authors   acknowledge the financial support of the SFB~701 ``Spectral structures and topological methods in mathematics'' (Bielefeld University).

\end{document}